\documentclass[12pt, a4]{amsart}

\usepackage{latexsym}
\usepackage{amsmath}
\usepackage{amssymb}
\usepackage{amsbsy}
\usepackage[dvips]{graphicx,color}
\def\R {\mathbb{R}}
\def\Z {\mathbb{Z}}
\def\ZZ {\mathbb{Z}_{2}}
\def\C {\mathbb{C}}
\def\N {\mathbb{N}}
\providecommand{\SO}{\mathop{\rm SO}\nolimits}
\providecommand{\Int}{\mathop{\rm Int}\nolimits}
\title{Links and submersions to the plane\\ on an open 3-manifold}

\author{Shigeaki Miyoshi}
\address{Department of Mathematics\\ Chuo University\\
   1-13-27 Kasuga Bunkyo-ku, Tokyo\\ 112-8551, Japan}
\email{miyoshi@math.chuo-u.ac.jp}

\thanks{
The author is supported by Grant-in-Aid for Scientific Research 23540106.
}

\newtheorem{lemma}{Lemma}               % 2nd argument is what is printed

\newtheorem{CrrctThm}{Theorem}

\newtheorem*{MainThm}{Main Theorem}
\newtheorem{AppendixThm}{Theorem}[section]
\newtheorem{AppendixLem}[AppendixThm]{Lemma}
\newtheorem*{AppendixClaim}{Claim}
\theoremstyle{definition}
\newtheorem*{defun}{Definition}        % Unnumbered environment for definition

\newtheorem{remark}{Remark}

\newtheorem{claim}{Claim}

\newtheorem*{acknowledgements}{Acknowledgements}

\newtheorem{AppendixRem}{Remark}[section]

\begin{document}

\baselineskip=18pt

\begin{abstract}    % abstract
We study the realization problem which asks if a given oriented link in
 an open 3-manifold can be realized as a fiber of a submersion to the
 Euclidean plane.  We correct the results obtained before by the author
 which contains an error and certain imperfection, and investigate a 
 necessary and sufficient condition for the realization in the words of
 well-known invariants.  We obtain the condition expressed by 
 the first homology group with mod 2 coefficient.  
\end{abstract}

\maketitle

%%%%%%%%%%%%%%%%%%%%   Start of main body of article
\section{Introduction} 
\label{intro}

The purpose of this paper is two-fold.  First, we will correct a theorem
in \cite{Top_paper} by the author, and second, study a problem which
arises from the correction.  G. Hector and D. Peralta-Salas found 
the error in the theorem in \cite{Top_paper} and informed the author about it.
Moreover, they studied comprehensively the
realization problem which asks if a manifold can be embedded in another
manifold so that it is also a fiber of a submersion to Euclidean
space (see \cite{Hector-Peralta}).  

Before stating the correct theorem, we prepare some notions.  We mostly
work in the smooth ($C^{\infty}$) category in this paper.  Suppose that
$M$ is an open oriented $3$-manifold and $L$ is an oriented
$n$-component link in $M$.  In this paper we say that a manifold is {\em
open} if the boundary is empty and no component is compact.  Let $N(L)$
denote a small tubular neighborhood of $L$.  A {\em framing} $\nu$ of
$L$ is meant to be an embedding $\nu :\bigsqcup_{j=1}^{n}(S^{1}\times
D^{2})_{j}\rightarrow M$ onto  $N(L)$ which maps the cores
$\bigsqcup_{j=1}^{n}(S^{1}\times\{ 0\})_{j}$ onto
$L=\bigsqcup_{j=1}^{n}L_{j}$.  Here, $D^{2}$ denotes the unit disk in
$\C$ and $S^{1}=\partial D^{2}$.  We assume that any framings and their
restrictions to the cores are orientation-preserving.  We note that a
framing of $L$ induces a tangential framing of $TN(L)(=TM|N(L))$ and
vice versa.  Here, a tangential framing means a choice of a
trivialization of $TN(L)\cong N(L)\times (\R\times\R^{2})$ with $TN(L)|L
= TL\oplus TL^{\bot}\cong (L\times\R)\oplus (L\times\R^{2})$ where
$TL^{\bot}$ denotes a normal bundle to $TL$.  

%%%%%%%%%% Definition of preferred framing %%%%%%%%%%%%%%%%%%%%%
\begin{defun}
 Suppose that $L$ represents the null-class in the locally finite
 homology group $H^{\infty}_{1}(M;\Z)$, i.e., the homology group of  
 locally finite (possibly) infinite chains.  A framing $\nu$
 of $L$ is said to be {\em preferred} (or {\em null-homologous}) if the
 union of the longitudes $\nu (\bigsqcup_{j=1}^{n}(S^{1}\times\{1 \})_{j})$
 represents the null-class in $H_{1}^{\infty}(M\setminus\Int N(L);\Z )$.
 We call $\nu (\bigsqcup_{j=1}^{n}(S^{1}\times\{1 \})_{j})$ the {\em preferred
 longitudes} of $N(L)$ with respect to the preferred framing $\nu$.  
\end{defun}
%%%%%%%%%%%%%%%%%%%%%%%%%%%%%%%%%%%%%%%%%%%%%%%%%%%%%%%%%%%%%%%%
%%%%%%%%%% Remark 1, 2 %%%%%%%%%%%%%%%%%%%%%%%%%%%%%%%%%%%%%%%%%
\begin{remark}
 If $L$ represents the null-class in $H_{1}^{\infty}(M;\Z )$, then there
 exists a preferred framing of $L$.  In fact, there exists an oriented
 (possibly non-compact) surface in $M$ bounded by $L$.  Choosing such a
 surface $S$, we have a framing of $L$ whose longitudes are
 $S\cap\partial N(L)$.  
\end{remark}
\begin{remark}
 Note that a preferred framing is not unique in general.  In fact, in
 the case of the core circle of the open solid torus, every framing is
 preferred.  Nevertheless, we invoke the terminology of {\em preferred longitudes} in
 \cite{Rolfsen}.  
\end{remark}
%%%%%%%%%%%%%%%%%%%%%%%%%%%%%%%%%%%%%%%%%%%%%%%%%%% Remark 1,2 %%%%%%

The correct theorem is the following.

%%%%%%%%%% Theorem A %%%%%%%%%%%%%%%%%%%%%%%%%%%%%%%%%%%%%%%%%%%%
 \begin{CrrctThm}\label{correct_thm}
For an oriented link $L$ in an open oriented $3$-manifold $M$, the
  following conditions are equivalent:
\begin{enumerate}
\item\label{realization}
there exists a submersion $\varphi : M\rightarrow\R^{2}$ such that up to
     isotopy the preimage $\varphi^{-1}(0)$ of the origin is $L$ and
     $\varphi$ maps the transverse orientation of $L$ to the standard
     orientation of $\R^{2}$, i.e., for any small disk $D$ transverse to
     $L$ with the orientation induced from those of $M$ and $L$, the
     restriction $\varphi |D$ preserves the orientation, and
\item\label{condition2}
the cycle $L$ represents the null-class in the locally finite homology
     group $H_{1}^{\infty}(M;\Z )$ and there exists a preferred
     framing of $L$ whose tangential framing is the restriction of some
     trivialization of $TM$.  
\end{enumerate}
\end{CrrctThm}
%%%%%%%%%%%%%%%%%%%%%%%%%%%%%%%%%%%%%%%%%%%%%%%%%%%%%%%%% Theorem A %%%
\noindent
Theorem \ref{correct_thm} is also a consequence of Theorem 2.4.2 in
\cite{Hector-Peralta}.  In the original incorrect theorem (Theorem 1 in
\cite{Top_paper}) the above extension condition of the framing in
(\ref{condition2}) is missing. Here, we explain briefly how it fills the gap in
the original proof.  If (\ref{realization}) holds, then the canonical
trivialization of the tangent bundle of $\R^{2}$ is pulled back to a
normal bundle to the fibers.  With a trivialization of the tangent
bundle to the fibers, it determines a trivialization of $TM$ which
restricts to a tangential framing of $L$.  The projection map from $N(L)$
onto the meridian disk determined by this framing must coincide with
the submersion restricted to $N(L)$.  Conversely, by the assumption that
the framing of $L$ is preferred,  the projection map
$N(L)\approx\bigsqcup_{j=1}^{n}(S^{1}\times D^{2})_{j}\rightarrow D^{2}$
extends to a map $(M, M\setminus\Int N(L))\rightarrow (\R^{2},
\R^{2}\setminus\Int D^{2})$ and moreover an extension of the
(tangential) framing of $L$ to $M$ ensures that we can take a submersion
$M\rightarrow\R^{2}$ as the extended map.  This is an application of the
{\em h-principle}, in this case, A. Phillips' submersion classification
theory \cite{Phillips}.  In the proof in \cite{Top_paper}, it is only
shown that an extension as a map exists since the framing of $L$ is
preferred.  However, in order to apply Phillips' theory to have an
extended submersion, we need the requirement of the simultaneous
extension of the tangential framing of $L$ and the projection map on
$N(L)$ to the whole manifold $M$.  

We note that Theorem 2 in \cite{Top_paper} is correct even though the
proof in \cite{Top_paper} is not completed.  
%%%%%%%%%%% Theorem B %%%%%%%%%%%%%%%%%%%%%%%%%%%%%%%%%%%%%%%%%%%%%%
\begin{CrrctThm}[Theorem 2 in \cite{Top_paper}]\label{thm_B}
For any link $L$ in an open orientable $3$-manifold, there is a submersion
 $\varphi : M\rightarrow\R^{2}$ such that up to isotopy the union of
 compact components of $\varphi^{-1}(0)$ is $L$.  
\end{CrrctThm}
%%%%%%%%%%%%%%%%%%%%%%%%%%%%%%%%%%%%%%%%%%%%%%%%%%%%%%% Theorem B %%%
\noindent
In order to prove this theorem, we have to choose a (tangential) framing
of $L$ which is the restriction of some trivialization over the whole
manifold $M$.  This can be always done by twisting a framing once around
the meridional direction if necessary.  Note that we need not to require
that the framing is preferred here.  This observation is missing in the proof in
\cite{Top_paper}.  
Theorem \ref{thm_B} is also proved in Application 2.3.8
in \cite{Hector-Peralta}.  
We will give the proof of Theorem \ref{correct_thm} and \ref{thm_B} in
Section \ref{appendix} as an appendix.  

As a consequence of the correction, there arises a question to find a
criterion for a link to be a fiber of a submersion to the plane in the words of
well-known invariants.  We will answer to this question for the case of a
knot.  Suppose that $M$ is an open oriented 3-manifold and $K$ is an
oriented knot in $M$.  For the simplicity, we say that $K$ is {\em
realizable} if $K$ satisfies the condition (\ref{realization}) in Theorem
\ref{correct_thm}. 

The following is the main theorem of this paper.
%%%%%%%%%% Main Theorem %%%%%%%%%%%%%%%%%%%%%%%%%%%%%%%%%%%%%%%%%%%%%%%%%
\begin{MainThm}\label{main_thm}
Assume that $K$ represents the null-class in $H_{1}^{\infty}(M;\Z )$.
 Then, $K$ is realizable if and only if $K$ represents a non-zero class in
 $H_{1}(M;\ZZ )$, where $\Z_{2}=\Z/2\Z$.  
\end{MainThm}
%%%%%%%%%%%%%%%%%%%%%%%%%%%%%%%%%%%%%%%%%%%%%%%%%%%%%%% Main Theorem %%%%%

\noindent
In order to prove the Main Theorem, it suffices to show the following
two claims. Let $\kappa$ denote the homology class $\iota_{\ast}([K])\in
H_{1}(M;\Z )$, where $\iota :K\hookrightarrow M$ is the inclusion map and
$[K]$ denotes the fundamental class of $K$.  
Also, let $\kappa_{(2)}$ denote the $\ZZ$-reduction of $\kappa$ in
$H_{1}(M;\ZZ )$.  
%%%%%%% Claim 1 and 2 %%%%%%%%%%%%%%%%%%%%%%%%%%%%%%%%%%%%%%%%%%%%%%%%%%%
\begin{claim}\label{claim_if}
 If $\kappa_{(2)}\neq 0$ then $K$ is realizable.
\end{claim}
\begin{claim}\label{claim_onlyif}
 If $\kappa_{(2)}=0$ then $K$ is not realizable.
\end{claim}
%%%%%%%%%%%%%%%%%%%%%%%%%%%%%%%%%%%%%%%%%%%%%%%%%%%%%%%%%%%%%%%%%%%%%%%%%

%%%%%%%%% Remark 3, 4 %%%%%%%%%%%%%%%%%%%%%%%%%%%%%%%%%%%%%%%%%%%%%%%%%%%%%
\begin{remark}
 As mentioned earlier, G. Hector and D. Peralta-Salas
 \cite{Hector-Peralta} studied this kind of realization problem in the
 more general dimensions and setting.  As one application of their
 theory, they obtained a characterization for a link in $\R^{3}$ to be
 realizable and in particular they 
 showed that no knot in $\R^{3}$ is realizable.  One may
 consider that the Main Theorem generalizes the result.  
\end{remark}
\begin{remark}
 In the case of links, the argument will be a rather complicated nuisance.
 It might be just a technicality, nevertheless we omit here the consideration
 in the case of links at all.  The complete research including the
 general case of links should be done in the sequel.  
\end{remark}
%%%%%%%%%%%%%%%%%%%%%%%%%%%%%%%%%%%%%%%%%%%%%%%%%%%%%%%%% Remark 3, 4 %%%
In Section \ref{proof_claim_1}, we describe the notion of tangential
framings of oriented knots from the homotopical viewpoint, and  prove
Claim \ref{claim_if}.  In Section \ref{proof_claim_2}, we study the
properties of framings of oriented knots and prove Claim
\ref{claim_onlyif}.  For the reader's convenience, we state a part of
Phillips' theory \cite{Phillips} which we need and give the proofs of
Theorem \ref{correct_thm} and \ref{thm_B} in Section \ref{appendix}, as
an appendix.

\section{Proof of Claim \ref{claim_if}}\label{proof_claim_1}
First, we express the notion of framings of oriented knots
in different words.  Suppose that $M$ is an oriented open $3$-manifold
and $K$ is an oriented knot in $M$.   We fix a trivialization $\Pi : TM\cong
M\times\R^{3}$ throughout the paper.  Suppose any framing $\nu :
S^{1}\times D^{2}\rightarrow N(K)$ is given.  Isotoping $\nu$ if
necessary, we may assume that each meridian disk $\nu
(\{\mathrm{exp}t\sqrt{-1}\}\times D^{2})$ is normal to $K$ with respect
to the metric induced by $\Pi$.  We will define a map $f_{\nu}:
K\rightarrow\SO(3)$ which is an alternate of $\nu$ under $\Pi$
as follows.  For any point $p\in K$, let $(v_{1}(p), v_{2}(p),
v_{3}(p))$ be the orthonormal frame of $T_{p}M=\R^{3}$ determined by the
tangent bundle and the normal bundle to $K$.  Precisely, $v_{1}(p)$ is
the unit tangent vector to $K$, $v_{2}(p)$ is the unit normal vector
determined by $\frac{\partial}{\partial x}$, and another unit normal
vector $v_{3}(p)$ is chosen by the orientation of $M$.  Here, we write
$z=x+y\sqrt{-1}\in D^{2}$.  Thus, with respect to $\Pi$, this orthonormal
frame $(v_{1}(p), v_{2}(p), v_{3}(p))$ can be expressed as a special
orthogonal matrix.  We define $f_{\nu}(p):= (v_{1}(p), v_{2}(p),
v_{3}(p))\in\SO(3)$.  

%%%%%%%%%%% Definition of \sigma-framing %%%%%%%%%%%%%%%%%%%%%%%%%%%%%%%%%%
\begin{defun}
We call the resulting map $f_{\nu}:K\rightarrow\SO(3)$ a {\em
 $\sigma$-framing} of $K$ with respect to $\nu$. 
\end{defun}
%%%%%%%%%%%%%%%%%%%%%%%%%%%%%%%%%%%%%%%%%%%%%%% Def. of \sigma-framing %%%%

%%%%%%%%%%% Remark 5 %%%%%%%%%%%%%%%%%%%%%%%%%%%%%%%%%%%%%%%%%%%%%%%%%%%%%%
\begin{remark}
In fact, a $\sigma$-framing is the component of a cross section of the
frame bundle $\mathrm{Fr}(TM)\cong M\times\SO(3)$ associated with $TM$ with the
trivialization induced by $\Pi$.   
\end{remark}
%%%%%%%%%%%%%%%%%%%%%%%%%%%%%%%%%%%%%%%%%%%%%%%%%%%%%%%%%%%% Rem.5 %%%%%%%%

It can be easily seen that under the parallelization $\Pi$ any map $K\rightarrow\SO(3)$
determines a framing $S^{1}\times D^{2}\rightarrow N(K)$ up to isotopy.
Thus, to choose a framing (up to isotopy) and to choose a $\sigma$-framing (up
to homotopy) are equivalent.  Moreover, since $[S^{1}, \SO
(3)]\cong\mathrm{Hom}(\pi_{1}(S^{1}), \pi_{1}(\SO (3)))\cong\pi_{1}(\SO
(3))\cong H_{1}(\SO(3);\Z )\cong\ZZ$ with appropriate choices of base
points, we may identify the homotopy class $[f_{\nu}]$ of a
$\sigma$-framing with its image $(f_{\nu})_{\ast}([K])\in H_{1}(\SO
(3);\Z )\cong\pi_{1}(\SO(3))$.  Note that a tangential framing of $K$ is
the restriction of some trivialization of $TM$ if and only if the
corresponding $\sigma$-framing of $K$ extends to $M$ as a map.  

Now we have the following criteria for the existence of an extension of a
$\sigma$-framing.

%%%%%% Lemma 1 %%%%%%%%%%%%%%%%%%%%%%%%%%%%%%%%%%%%%%%%%%%%%%%%%%%%%
\begin{lemma}\label{extension_lem}
 Let $M$ be an open orientable $3$-manifold and $K$ a knot in $M$.
 Suppose that a map $f:K\rightarrow\SO (3)$ is given.  Then the
 following are equivalent: 
\begin{enumerate}
\item
the map $f:K\rightarrow\SO (3)$ extends to a map $M\rightarrow\SO (3)$, 
\item
the induced homomorphism $f_{\ast}:\pi_{1}(K)\rightarrow\pi_{1}(\SO
     (3))$ extends to a homomorphism $\pi_{1}(M)\rightarrow \pi_{1}(\SO
     (3))$, and 
\item
the homomorphism $f_{\ast}:H_{1}(K;\Z )\rightarrow H_{1}(\SO (3);\Z )$
     extends to a homomorphism $H_{1}(M;\Z )\rightarrow H_{1}(\SO (3);\Z
     )$.  
\end{enumerate}
 Moreover, in the implication from $(2)$ or $(3)$ to $(1)$, the resulting
 extension map $M\rightarrow\SO(3)$ induces the given extended
 homomorphism.  
\end{lemma}
%%%%%%%%%%%%%%%%%%%%%%%%%%%%%%%%%%%%%%%%%%%%%%%%%%%%%%%% Lemma 1 %%%%
%%%%%%%% Proof of Lemma 1 %%%%%%%%%%%%%%%%%%%%%%%%%%%%%%%%%%%%%%%%%%%
\begin{proof}
 It is well known that an open orientable $3$-manifold is homotopy
 equivalent to a subcomplex of its $2$-skeleton (cf. \cite{Phillips},
 \cite{Whitehead} for example).  Therefore, by an
 elementary obstruction theory, the given map $f:K\rightarrow\SO (3)$ extends to
 a map $M\rightarrow\SO (3)$ if and only if the induced homomorphism
 $f_{\ast}:\pi_{1}(K)\rightarrow\pi_{1}(\SO (3))$ extends to a
 homomorphism $\pi_{1}(M)\rightarrow\pi_{1}(\SO (3))$, i.e., there is a
 homomorphism $\Phi:\pi_{1}(M)\rightarrow\pi_{1}(\SO (3))$ such that
 $\Phi\circ\iota_{\ast}=f_{\ast}$ where $\iota:K\hookrightarrow M$ is the
 inclusion.  Moreover, since $\pi_{1}(K)$ and $\pi_{1}(\SO (3))$ are
 Abelian it is equivalent to the condition that the homomorphism
 $f_{\ast}:H_{1}(K;\Z )\rightarrow H_{1}(\SO (3);\Z )$ extends to a
 homomorphism $H_{1}(M;\Z )\rightarrow H_{1}(\SO (3);\Z )$. 
\end{proof}

 Now, we show Claim \ref{claim_if}. 
\begin{proof}[Proof of Claim \ref{claim_if}]
 Suppose $f:K\rightarrow\SO(3)$ is a preferred $\sigma$-framing
 of $K$, i.e., the $\sigma$-framing associated with a preferred framing
 of $K$.  By Theorem \ref{correct_thm} and  Lemma \ref{extension_lem},
 it suffices to show that there exists a homomorphism $\Phi
 :H_{1}(M;\Z )\rightarrow H_{1}(\SO(3);\Z )$ which is an
 extension of the induced homomorphism $f_{\ast}:H_{1}(K;\Z )\rightarrow
 H_{1}(\SO(3);\Z )$.  If $f_{\ast}=0$ then the zero
 homomorphism is an extension.  Hence we assume $f_{\ast}\neq 0$,
 which implies $f_{\ast}$ is an epimorphism.  On the other hand, since
 $\kappa_{(2)}\neq 0$ the composition of the natural homomorphisms
 $H_{1}(K;\Z )\rightarrow H_{1}(M;\Z )\rightarrow H_{1}(M;\ZZ )$ is
 non-trivial.  Let $\langle\kappa_{(2)}\rangle$ denote its image.  Since
 $H_{1}(M;\ZZ )$ is a vector space over the field $\ZZ$, we have a
 projection onto the one-dimensional subspace $H_{1}(M;\ZZ
 )\rightarrow\langle\kappa_{(2)}\rangle$.  Let $\Psi$ denote the
 composition of the natural homomorphism $H_{1}(M;\Z )\rightarrow
 H_{1}(M;\ZZ )$ followed by this projection.  Let $\alpha
 :\langle\kappa_{(2)}\rangle\cong\ZZ$ and $\beta
 :H_{1}(\SO(3);\Z )\cong\ZZ$ be any isomorphisms.  Then the
 composition $\Phi =\beta^{-1}\circ\alpha\circ\Psi$ is the desired
 extension homomorphism. 
\end{proof}

\section{Lemmata and proof of Claim \ref{claim_onlyif}}\label{proof_claim_2}
In this section, we study some properties of $\sigma$-framings and
prove Claim \ref{claim_onlyif}.  The following lemma describes a
relation of ($\sigma$-)framings of two oriented knots which are homologous. 
%%%%%%%%% Lemma 2 %%%%%%%%%%%%%%%%%%%%%%%%%%%%%%%%%%%%%%%%%%%%%%%%%%%%%%%%%
\begin{lemma}\label{homologous_framings}
 Let $Z_{1}$ and $Z_{2}$ be oriented knots in $M$ and
 $\nu_{j}:S^{1}\times D^{2}\rightarrow N(Z_{j})$ their framings $(j=1, 2)$ .  
 Assume that there is a compact oriented surface $S$ in
 $M$ such that $\partial S=Z_{1}\sqcup (-Z_{2})$ and $S\cap\partial
 N(Z_{j})=\nu_{j}(S^{1}\times\{1\} )$, where $-Z_{2}$ denotes $Z_{2}$
 with the orientation reversed. Then the induced $\sigma$-framings
 $f_{\nu_{1}}$ and $f_{\nu_{2}}$ satisfy that
 $[f_{\nu_{1}}]=[f_{\nu_{2}}]\in\pi_{1}(\SO (3))$.  
\end{lemma}
%%%%%%%%%%%%%%%%%%%%%%%%%%%%%%%%%%%%%%%%%%%%%% Lemma 2 %%%%%%%%%%%%%%%%%%
\begin{proof}[Proof of Lemma \ref{homologous_framings}]
 We define a map $F:S\rightarrow \SO(3)$ as follows.  Choose a unit
 tangent vector field $v_{1}:S\rightarrow TS\subset TM|S$ such that
 $v_{1}|Z_{1}$ coincides with the unit vector field tangent to $Z_{1}$.
 Then we choose another vector field $v_{2}$ so that $(v_{1}, v_{2})$
 forms an orthonormal frame field of $S$.  Here, $v_{2}$ is chosen to be
 inward normal along $Z_{1}$.  Picking the normal unit vector field
 $v^{\bot}$ to $S$, we have a frame field $F=(v_{1}, v_{2},
 v^{\bot}):S\rightarrow\SO(3)$.  By the definition,
 $F|Z_{1}=f_{\nu_{1}}$.  Since the rotation number of $v_{1}|Z_{2}$
 along $Z_{2}$ is equal to the Euler characteristic $\chi (S)$ which is
 the minus twice the genus of $S$, we have
 $[F|Z_{2}]=[f_{\nu_{2}}]\in\pi_{1}(\SO(3))$.  Since $F|Z_{1}$ and
 $F|Z_{2}$ are homologous (or bordant) by $F$, we have
 $[f_{\nu_{1}}]=[F|Z_{1}]=[F|Z_{2}]=[f_{\nu_{2}}]$.  
\end{proof}

Next, we study an oriented knot whose homology class with $\ZZ$
coefficient is zero.   
First, we consider the ``double'' of a knot and study its framing.  
Let $J$ be any oriented knot in $M$ and
$\nu :S^{1}\times D^{2}\rightarrow N(J)$ any framing of $J$.  Let
$J_{d}$ be the $(2, 1)$-cable knot in $N(J)$.  For the clarity, we define
$J_{d}$ as follows.  Define $\tilde{L}$ to be a union of two parallel
lines in $\R\times D^{2}$ as 
\[
 \tilde{L}:=\{ (t, \pm\tfrac{1}{2}) |\ t\in\R, \}
\]
and $\tau$ to be a self-diffeomorphism on $\R\times D^{2}$ by $\tau (t,
z):=(t, \exp (\pi t\sqrt{-1})z)$.  Then the quotient of $(\R\times
D^{2}, \tau (\tilde{L}))$ by the $\Z$-action generated by the translation by $1$
on the $\R$-factor is a manifold pair $(S^{1}\times D^{2}, L)$ such that
$L$ is an oriented knot. 
Here, we identify $S^{1}=\R /\Z$.  
We define $J_{d}$ to be $\nu (L)$ and call it a
$(2, 1)$-{\em cable knot} of $J$ with respect to $\nu$.  

To define a natural framing of a $(2, 1)$-cable knot $J_{d}$ of $J$, we
consider an annulus in $S^{1}\times D^{2}$ defined as follows.  Let
$\tilde{A}$ be a union of two strips in $\R\times D^{2}$ defined by 
\[
 \tilde{A}:=\{ (t, \pm r) |\ t\in\R, \tfrac{1}{2}\leq r\leq 1\}
\]
and set $(S^{1}\times D^{2}, A):=(\R\times D^{2}, \tau (\tilde{A}))/\Z$
as the quotient by the translation.  Note that $\partial A-\partial
(S^{1}\times D^{2})=L$.  For any small tubular neighborhood $N(L)$ in
$\Int (S^{1}\times D^{2})$, there is a framing $\xi :S^{1}\times
D^{2}\rightarrow N(L)\subset S^{1}\times D^{2}$ of $L$ such that $\xi
(S^{1}\times\{ 1\})=\partial N(L)\cap A$.  We call the framing
$\nu\circ\xi :S^{1}\times D^{2}\rightarrow N(J_{d})=\nu (N(L))\subset
N(J)$ a {\em revolution framing} of the $(2, 1)$-cable knot $J_{d}(=\nu
(L))$ of $J$ with respect to $\nu$.  See Figure~\ref{fig_RevFraming}.  
\begin{figure}[h]
	\centering
	\includegraphics[height=6cm]{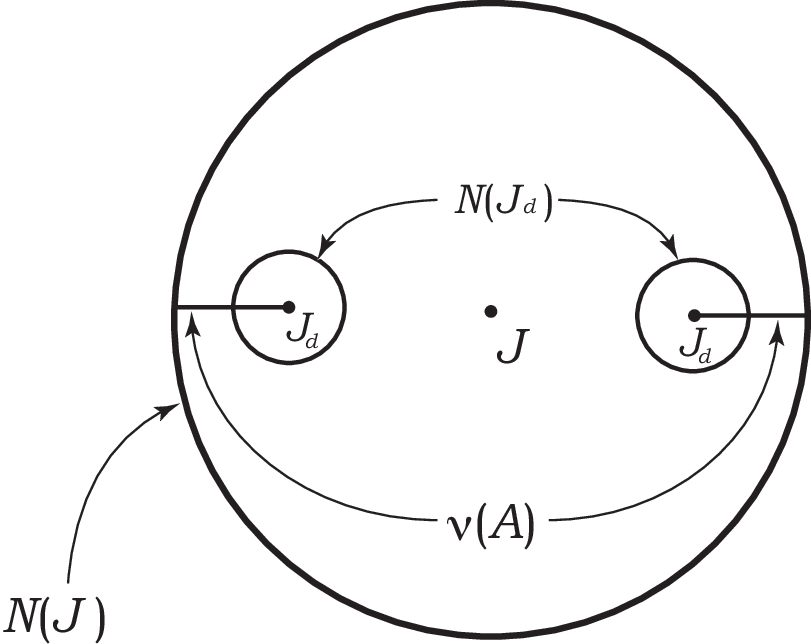}
	\caption{A revolution framing of $J_{d}$ at a section}
	\label{fig_RevFraming}
\end{figure}
Under the parallelization $\Pi$ of $M$, the revolution framing of
$J_{d}$ induces a {\em revolution $\sigma$-framing}
$J_{d}\rightarrow\SO(3)$.  

The following is a key lemma to the proof of Claim \ref{claim_onlyif}.  
%%%%%%%% Lemma 3((2,1)-framing ) %%%%%%%%%%%%%%%%%%%%%%%%%%%%%%%%%%%%%%% 
\begin{lemma}\label{(2,1)-framing}
 For any oriented knot $J$ in $M$ and any framing $\nu :S^{1}\times
 D^{2}\rightarrow N(J)$, the revolution $\sigma$-framing
 $f_{d}:J_{d}\rightarrow\SO(3)$ of the $(2,1)$-cable knot
 $J_{d}$ with respect to $\nu$ is not null-homotopic, i.e.,
 $[f_{d}]=1\in\ZZ\cong\pi_{1}(\SO(3))$.  
\end{lemma}
%%%%%%%%%%%%%%%%%%%%%%%%%%%%%%%%%%%%%%%%%%%% Lemma 3((2,1)framing) %%%%%
\begin{proof}
 Let $f_{\nu}:J\rightarrow\SO(3)$ be the $\sigma$-framing of $J$
 with respect to $\nu$.  Then along $J_{d}$ the frame $f_{d}(p)$ goes
 around twice in the longitudinal direction and rotates once in the
 meridional direction.  Thus $[f_{d}]=2[f_{\nu}]+1\equiv
 1\in\ZZ\cong\pi_{1}(\SO(3))$.  
\end{proof}

Next, suppose that $K$ is an oriented knot $K$ which represents the
null-class in $H_{1}^{\infty}(M;\Z )$. Recall that $\kappa_{(2)}\in
H_{1}(M;\ZZ )$ denotes the $\ZZ$-reduction of the homology class $\kappa
=\iota_{\ast}([K])\in H_{1}(M;\Z )$.  Let $\lambda$ be the homology class in
$H_{1}(M\setminus\Int N(K);\Z )$ represented
by a preferred longitude of $N(K)$.  Then we have the following. 
%%%%%%%%%%% Lemma 4(\lambda_(2)=0) %%%%%%%%%%%%%%%%%%%%%%%%%%%%%%%%%%%%%
\begin{lemma}\label{longitude_null}
 If $\kappa_{(2)}=0$, then the $\ZZ$-reduction $\lambda_{(2)}$ of
 $\lambda$ is zero.  
\end{lemma}
%%%%%%%%%%%%%%%%%%%%%%%%%%%%%%%%%%%%%%%%%%%%%% Lemma 4(\lambda_(2)=0) %%
\begin{proof}
 Let $E_{M}(K)$ denote the knot exterior $M\setminus\Int N(K)$.
 Consider the following diagram. 
\[
 \begin{array}{cccccccc}
 \rightarrow & H_{2}(M, E_{M}(K);\Z ) &
  \stackrel{\partial}{\rightarrow} & H_{1}(E_{M}(K);\Z ) &
  \stackrel{\iota_{\ast}}{\rightarrow} 
  & H_{1}(M;\Z ) & \rightarrow & 0\\
  & \downarrow & & \downarrow & & \downarrow & & \\
 \rightarrow & H_{2}(M, E_{M}(K);\ZZ ) &
  \stackrel{\partial}{\rightarrow} & H_{1}(E_{M}(K);\ZZ ) &
  \stackrel{\iota_{\ast}}{\rightarrow} & H_{1}(M;\ZZ ) & \rightarrow & 0 
 \end{array}
\]
 Here, the rows are homology exact sequences of the pair $(M, E_{M}(K))$ and
 the vertical arrows are natural homomorphisms.  Then
 $\iota_{\ast}(\lambda)=\kappa , \ 
 \iota_{\ast}(\lambda_{(2)})=\kappa_{(2)}$ and $\lambda ,\kappa$ are mapped down
 to $\lambda_{(2)}, \kappa_{(2)}$ respectively.  Since $\kappa_{(2)}=0$,
 there is $\eta\in H_{2}(M, E_{M}(K);\ZZ )$ such that $\partial\eta
 =\lambda_{(2)}$.  Note that $H_{2}(M, E_{M}(K);\ZZ )$ is isomorphic to
 $\ZZ$ generated by the meridian disk of $N(K)$.  Hence $\partial\eta$ is
 represented by the meridian loop or equal to zero.  However, the
 meridian loop intersects exactly once with a (locally finite) relative cycle in
 $(E_{M}(K), \partial E_{M}(K))$ bounded by the preferred longitude of
 $N(K)$, the representative cycle of $\lambda$.  Thus $\partial\eta
 =\lambda_{(2)}$ must be zero.  
\end{proof}

The following is another key lemma which describes a normal form of the
knot which satisfies the hypothesis of Claim \ref{claim_onlyif}.  
%%%%%%% Lemma 5(Transportation from K to a (2, 1)-cable) %%%%%%%%%%%%%%%%%
\begin{lemma}\label{transportation_lemma}
 Suppose that an oriented knot $K$ represents the null-class in
 $H_{1}^{\infty}(M;\Z )$ and $\kappa_{(2)}=0$.  We fix a preferred
 framing $\nu :S^{1}\times D^{2}\rightarrow N(K)$.  Then, there
 exists an oriented knot $Z\subset M$ such that the $(2, 1)$-cable knot
 $Z_{d}$ of $Z$ is homologous to $K$ and
 $[f_{d}]=[f_{\nu}]\in\pi_{1}(\SO(3))$, where
 $f_{d}:Z_{d}\rightarrow\SO(3)$ is the revolution
 $\sigma$-framing of $Z_{d}$ with respect to some framing of $Z$.  
\end{lemma}
%%%%%%%%%%%%%%%%%%%%%%%%%%%%%%%%%%%%%%%%%%%%%% Lemma 5 %%%%%%%%%%%%%%%%%
\begin{proof}
 Let $L$ denote the preferred longitude of $N(K)$ with respect to
 $\nu$ and $\lambda$ the homology class in $H_{1}(E_{M}(K);\Z )$
 represented by $L$.  Then we have $\lambda_{(2)}=0$ by Lemma
 \ref{longitude_null} using the assumption $\kappa_{(2)}=0$.  
 Considering the Bockstein homology exact sequence with respect to
 $0\rightarrow\Z\stackrel{\times
 2}{\rightarrow}\Z\rightarrow\ZZ\rightarrow 0$, we have $\lambda =2\zeta$
 for some $\zeta\in H_{1}(E_{M}(K);\Z )$.  Choose a
 representative cycle (an oriented knot) $Z\subset E_{M}(K)$ of $\zeta$ and a
 framing of $Z$.  Since $\lambda =2\zeta$, there exists an immersed
 oriented surface bounded by $L$ and ``twice of $-Z$''.  Precisely, there
 exists an immersion $h$ of compact oriented surface $S$ into $E_{M}(K)$
 which maps $\Int S$ into $\Int E_{M}(K)$ with the
 following properties.  The boundary $\partial S$ is decomposed 
 into two parts: $\partial_{+}S\sqcup\partial_{-}S$, where
 $h(\partial_{+}S)=L$ and $h(\partial_{-}S)=-Z$ which means
 $h|\partial_{-}S$ is orientation-reversing.  The immersion $h$ is an
 embedding away from $\partial_{-}S$ and $h|\partial_{-}S$ is a
 (possibly trivial) two-fold covering onto $Z$.  Moreover, $h(S)\cap
 N(Z)$ is homeomorphic (in fact diffeomorphic away from $Z$) to the
 mapping cylinder of the two-fold covering map $h|\partial_{-}S$.  We
 may assume that $h(S)$ and $\partial N(Z)$ are transverse to each other
 and $h(S)\cap\partial N(Z)$ is a circle or a union of two parallel
 circles in $\partial N(Z)$.  By choosing another framing of $N(Z)$ if
 necessary, we may assume that $h(S)\cap\partial N(Z)$ is the $(2,
 1)$-curve or a union of two $(1, 0)$-curves with respect to the chosen
 framing of $N(Z)$ restricted to $\partial N(Z)$.  

 If $h(S)\cap\partial N(Z)$ is the $(2, 1)$-curve, then attaching $A$ to
 $h(S)-\Int N(Z)$ along $h(S)\cap\partial N(Z)$ we have an
 embedded surface in $E_{M}(K)$ bounded by $L$ and $-Z_{d}$.  Here, $A$
 is the annulus in $N(Z)$ bounded by a $(2, 1)$-curve in
 $\partial N(Z)$ and $Z_{d}$, which is defined in the definition of the 
 revolution framing of $Z_{d}$, and the orientation of $A$ is determined 
 by $-Z_{d}$.  Since $L$ is isotopic to $K$ in $N(K)$, we have a compact 
 oriented surface in $M$ bounded by $K$ and $-Z_{d}$.  It follows that $K$ and
 $Z_{d}$ are homologous and $[f_{\nu}]=[f_{d}]$ by Lemma
 \ref{homologous_framings}.  

 In the case that $h(S)\cap\partial N(Z)$ is two $(1, 0)$-curves, we
 modify $h(S)$ in $N(Z)$ as follows.  Consider the concentric tubular
 neighborhood $N_{1/2}(Z)\subset N(Z)$ where meridian disks are of
 radius $\frac{1}{2}$ of the meridian disks of $N(Z)$.  We set a $(2,
 1)$-cable knot $Z_{d}$ on $\partial N_{1/2}(Z)$ with respect to the
 framing of $N(Z)$ and we will construct a compact oriented surface $B$
 in $N(Z)$ bounded by $h(S)\cap\partial N(Z)$ and another longitude of
 $N(Z_{d})$.  

\begin{figure}[h]
	\centering
	\includegraphics[height=6cm]{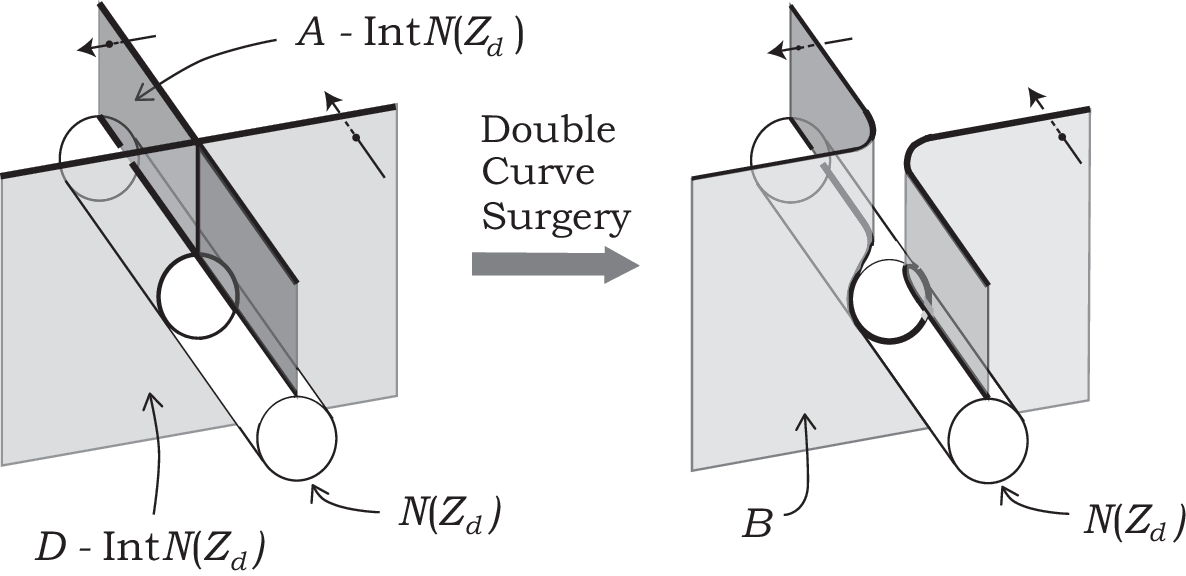}
	\caption{Double curve surgery on $(A\cup (-D))-\Int N(Z_{d})$}
	\label{fig_DCS}
\end{figure}

 First, let $D$ be a meridian disk of $N(Z)$ with the orientation
 induced from $S^{1}\times D^{2}$ by the framing and $A$ the annulus in $N(Z)$
 which determines the revolution framing of $Z_{d}$ as above.  Fix a
 small tubular neighborhood $N(Z_{d})$ of $Z_{d}$.  Then performing a
 double curve surgery on $(A\cup (-D))-\Int N(Z_{d})$, we obtain an
 oriented surface $B$ (see Figure~\ref{fig_DCS}).  By the construction,
 $B\cap\partial N(Z)$ is a union of two $(1, 0)$-curves and
 $B\cap\partial N(Z_{d})$ is the longitude of the revolution framing
 with two twists corresponding to two intersection points between $D$
 and $Z_{d}$.  Since on $\partial N(Z)$ two curves $h(S)\cap\partial N(Z)$
 and $B\cap\partial N(Z)$ are isotopic, we can attach $B$ to 
 $h(S)-\Int N(Z)$ along $h(S)\cap\partial N(Z)$.  Let $\Sigma$
 denote the resulting surface: $\Sigma = (h(S)-\Int N(Z))\cup B$.
 Then $\Sigma$ is a compact
 oriented proper surface in $E_{M}(K)\setminus\Int N(Z_{d})$ and
 $\partial\Sigma = L\sqcup (-B\cap \partial N(Z_{d}))$. 
 Since $L$ (resp. $B\cap \partial N(Z_{d})$) is isotopic to $K$ (resp. $Z_{d}$)
 in $N(K)$ (resp. $N(Z_{d})$), $K$ and $Z_{d}$ are homologous.  
 On the other hand, the framing
 $\xi :S^{1}\times D^{2}\rightarrow N(Z_{d})$ such that
 $\xi (S^{1}\times\{ 1\})=-\Sigma\cap\partial N(Z_{d})(=-B\cap\partial
 N(Z_{d}))$ determines a 
 $\sigma$-framing $f_{\xi}:Z_{d}\rightarrow\SO(3)$.  By Lemma
 \ref{homologous_framings}, we have $[f_{\nu}]=[f_{\xi}]$.  Moreover, as
 noted above, the longitude with respect to $\xi$ is the longitude of the
 revolution framing with two meridional twists.  Hence we have
 $[f_{\xi}]=[f_{d}]$.  Consequently, we have $[f_{d}]=[f_{\nu}]$.  
\end{proof}

Now we can prove Claim \ref{claim_onlyif}.  
%%%%%%%%% Proof of Claim 2 %%%%%%%%%%%%%%%%%%%%%%%%%%%%%%%%%%%%%%%%%%%%
\begin{proof}[Proof of Claim \ref{claim_onlyif}]
 Suppose that $f:K\rightarrow\SO(3)$ is any preferred
 $\sigma$-framing of $K$.  In view of Theorem \ref{correct_thm} and
 Lemma \ref{extension_lem}, we only have to show that the induced
 homomorphism $f_{\ast}:H_{1}(K;\Z )\rightarrow H_{1}(\SO(3);\Z
 )$ never extends to $H_{1}(M;\Z )$.  On the contrary to the
 conclusion, we assume that there is a homomorphism $\Phi :H_{1}(M;\Z
 )\rightarrow H_{1}(\SO(3);\Z )$ such that
 $\Phi\circ\iota_{\ast}=f_{\ast}$,  where $\iota :K\hookrightarrow M$
 denotes the inclusion.  Since $\kappa_{(2)}=0$, we have $\Phi (\kappa
 )=0$.  On the other hand, by Lemma \ref{(2,1)-framing} and
 \ref{transportation_lemma}, $f_{\ast}([K])=[f]=1$.  Therefore, we have 
 $0=\Phi (\kappa )=\Phi\circ\iota_{\ast}([K])=f_{\ast}([K])=1$, a
 contradiction. 
\end{proof}

\section{Appendix: Proofs of Theorem \ref{correct_thm} and \ref{thm_B}}
\label{appendix}
In order to prove Theorem \ref{correct_thm} and \ref{thm_B}, we review
Phillips' submersion classification theory \cite{Phillips}.  Let $X$ and
$Y$ be manifolds.  We assume that $\mathrm{dim}X\geq\mathrm{dim}Y$ in
the following.  The space of all submersions from $X$ to $Y$ is
denoted by $\mathrm{Sbm}(X, Y)$ and the space of all vector bundle
morphisms from $TX$ to $TY$ whose restriction to each fiber the
maximal rank by $\mathrm{Max}(TX, TY)$.  Here, $\mathrm{Sbm}(X, Y)$ and
$\mathrm{Max}(TX, TY)$ are endowed with $C^{1}$-compact-open topology
and $C^{0}$-compact-open topology, respectively.  If $X$ has a non-empty
boundary, we impose no other condition on the boundary.  The essence of
the Phillips' theory is the following theorem.  
%%%%%% Phillips' theorem (handle decomp) %%%%%%%%%%%%%%%%%%%%%%%%%%%%%%%%
\begin{AppendixThm}[Phillips \cite{Phillips}]\label{Phillips_HD_thm}
 If $X$ has a handle decomposition with (possibly countably infinitely
 many) handles of indices less than $\mathrm{dim}X$, then the
 differential map $d:\mathrm{Sbm}(X, Y)\rightarrow\mathrm{Max}(TX, TY)$
 is a weak homotopy equivalence.
\end{AppendixThm}
%%%%%%%%%%%%%%%%%%%%%%%%%%%%%%%%%%%% Phillips' theorem (handle decomp) %%
\noindent
Since an open manifold has a handle decomposition with (countably
infinitely many) handles of indices less than the dimension of the
manifold, we have the following theorem as a corollary.  
%%%%%%% Phillips' theorem (open mfd) %%%%%%%%%%%%%%%%%%%%%%%%%%%%%%%%%%%%
\begin{AppendixThm}[Phillips \cite{Phillips}]\label{Phillips_OM_thm}
 If $X$ is an open manifold, then the differential map
 $d:\mathrm{Sbm}(X, Y)\rightarrow\mathrm{Max}(TX, TY)$ is a weak
 homotopy equivalence.  
\end{AppendixThm}
%%%%%%%%%%%%%%%%%%%%%%%%%%%%%%%%%%% Phillips' theorem (open mfd) %%%%%%%%
\noindent
In the proof of Theorem \ref{Phillips_HD_thm} (and
\ref{Phillips_OM_thm}), the following are key lemmata.  
%%%%%%% Phillips' lemma (disk case) %%%%%%%%%%%%%%%%%%%%%%%%%%%%%%%%%%%%%
\begin{AppendixLem}\label{PhillipsLem_D}
 The differential map
 $d:\mathrm{Sbm}(D, Y)\rightarrow\mathrm{Max}(TD, TY)$ is a weak
 homotopy equivalence, where $D$ denotes a disk of dimension
 $\mathrm{dim}X$.   
\end{AppendixLem}
%%%%%%%%%%%%%%%%%%%%%%%%%%%%%%%%%%%%%%%% Phillips' lemma (disk case) %%%
%%%%%%% Phillips' lemma (hanndle attached) %%%%%%%%%%%%%%%%%%%%%%%%%%%%%
\begin{AppendixLem}\label{PhillipsLem_H}
 Let $V$ be a compact manifold with $\mathrm{dim}V=\mathrm{dim}X$.
 Suppose $W$ is obtained by attaching a handle of index less than
 $\mathrm{dim}V$.  Then, the restriction maps $\rho :\mathrm{Sbm}(W,
 Y)\rightarrow\mathrm{Sbm (V, Y)}$ and $\rho :\mathrm{Max}(TW,
 TY)\rightarrow\mathrm{Max}(TV, TY)$ are fibrations. 
\end{AppendixLem}
%%%%%%%%%%%%%%%%%%%%%%%%%%%%%% Phillips' lemma (handle attached) %%%%%%%
\noindent
The proof of Theorem \ref{Phillips_HD_thm} is carried out by starting
with Lemma \ref{PhillipsLem_D}, applying Lemma \ref{PhillipsLem_H}
handle by handle, and an inverse limit argument.  We refer the reader to 
\cite{Phillips} for the details.  Applying the inverse limit argument in
the proof of Theorem \ref{Phillips_HD_thm} and \ref{Phillips_OM_thm}, 
we have the following. 
%%%%%% Phillips' byproduct (inverse limit) %%%%%%%%%%%%%%%%%%%%%%%%%%%%
\begin{AppendixLem}\label{Phillips_byproduct}
 Let $W$ be a codimension $0$ compact submanifold of an open manifold
 $X$.  Then, the restriction maps $\rho :\mathrm{Sbm}(X,
 Y)\rightarrow\mathrm{Sbm (W, Y)}$ and $\rho :\mathrm{Max}(TX,
 TY)\rightarrow\mathrm{Max}(TW, TY)$ are fibrations.  
\end{AppendixLem}
%%%%%%%%%%%%%%%%%%%%%%%%%%%%%%%%%%%%%%%%%%%%%%%%%%%%%%%%%%%%%%%%%%%%%%%%
\noindent
To be precise, in the literature an open manifold could have a non-empty
boundary.  Thus, the following lemma might be in fact contained in Lemma
\ref{Phillips_byproduct}, however, we give it here as a precise
statement we need in the proof of Theorem \ref{correct_thm}.  
%%%%%%%% Phillips' fibration lemma %%%%%%%%%%%%%%%%%%%%%%%%%%%%%%%%%%%%%
\begin{AppendixLem}\label{fibration_lemma}
 Suppose that $X$ is a manifold with no compact component and $\partial
 X\neq\emptyset$.  Let $W$ be a codimension $0$ compact submanifold of
 $X$ such that $\partial X\subset\partial W$.  Then, the restriction
 maps $\rho :\mathrm{Sbm}(X, Y)\rightarrow\mathrm{Sbm (W, Y)}$ and $\rho
 :\mathrm{Max}(TX, TY)\rightarrow\mathrm{Max}(TW, TY)$ are fibrations.  
\end{AppendixLem}

Now, we can prove Theorem \ref{correct_thm}.  We add a correct
consideration on the trivialization of the tangent bundles, however,
we mostly follow the proof in \cite{Top_paper}.  
%%%%%%%%%%%%%%%%%%%%%%%%%%%%%%%%%%%%%%%%%%%%%%%%%%%%%%%%%%%%%%%%%%%%%%%%%%
%%%%%%% Proof of Theorem A %%%%%%%%%%%%%%%%%%%%%%%%%%%%%%%%%%%%%%%%%%%%%%%
%%%%%%%%%%%%%%%%%%%%%%%%%%%%%%%%%%%%%%%%%%%%%%%%%%%%%%%%%%%%%%%%%%%%%%%%%%
\begin{proof}[Proof of Theorem \ref{correct_thm}]
 Assume that (\ref{realization}) holds.  Then the preimage by $\varphi$ of
 a semiline starting from the origin to the end of $\R^{2}$ is a surface
 in $M$ which is bounded by $\varphi^{-1}(0)=L$.  By the condition of
 $\varphi$ on the transverse orientation to $L$, we may choose the
 orientation on the surface so that $L$ represents the null-class in the
 locally finite homology group.  Moreover, as mentioned in the Introduction,
 $\varphi$ determines a trivialization of $TM$ which restricts to a
 tangential framing of $L$.  The projection with respect to the framing
 associated with this tangential framing of $L$ coincides with $\varphi$
 near $L$.  Thus, (\ref{condition2}) holds.  

 Next, assume that (\ref{condition2}) holds.  Choose a framing $\nu
 :\bigsqcup_{j=1}^{n}(S^{1}\times D^{2})_{j}\rightarrow N(L)$ which is
 preferred and suppose there exists a trivialization of $TM$ which
 restricts to the trivialization of $TN(L)$ determined by the tangential
 framing induced by $\nu$.  For $0\leq r\leq 1$, set $D^{2}(r):=\{ z\in\C\
 |\ |z|\leq r\}$ and $N_{r}(L):=\nu (\bigsqcup_{j=1}^{n}(S^{1}\times
 D^{2}(r))_{j})$.  Define $\pi :N(L)\rightarrow D^{2}\subset\C =\R^{2}$ to
 be the composition $\mathrm{pr}\circ\nu^{-1}$, where
 $\mathrm{pr}:\bigsqcup_{j=1}^{n}(S^{1}\times D^{2})_{j}\rightarrow
 D^{2}$ is the natural projection onto a single disk.  Set
 $X:=M\setminus\Int N_{1/2}(L)$ and $W:=N(L)\setminus\Int N_{1/2}(L)$.
 Note that $\partial X=\partial N_{1/2}(L)$ since $\partial
 M=\emptyset$.  

 We consider the following commutative diagram consisting of the
 differential maps $d$ and the restriction maps $\rho$.  
\begin{equation}\label{h-principle_diagram}
 \begin{array}{ccc}
  \mathrm{Sbm}(X, C) & \stackrel{d}{\rightarrow} & \mathrm{Max}(TX,
   TC)\\
  \rho\downarrow & & \rho\downarrow\\
  \mathrm{Sbm}(W, C) & \stackrel{d}{\rightarrow} &
   \mathrm{Max}(TW, TC) 
 \end{array}
\end{equation}
 where $C$ denotes $\R^{2}\setminus\Int D^{2}(\frac{1}{2})$.  In the diagram
 the horizontal arrows are weak homotopy equivalences by Theorem
 \ref{Phillips_OM_thm} and \ref{Phillips_HD_thm}, and the vertical
 arrows are fibrations by Lemma \ref{fibration_lemma}.  The projection
 $\pi :N(L)\rightarrow D^{2}$ restricted to $W$, denoted by $\pi |W$,
 belongs to $\mathrm{Sbm}(W, C)$.  For $d(\pi |W)\in\mathrm{Max}(TW,
 TC)$, we have an extension as follows.  
%%%%%%%%% Extension Claim %%%%%%%%%%%%%%%%%%%%%%%%%%%%%%%%%%%%%%%%%%%%%%%%
\begin{AppendixClaim}
 There exists $\Phi\in\mathrm{Max}(TX, TC)$ such that $\rho
 (\Phi)=d(\pi |W)$.  
\end{AppendixClaim}
%%%%%%%%%%%%%%%%%%%%%%%%%%%%%%%%%%%%%%%%%%%%%%%%% Extension Claim %%%%%%%%
\begin{proof}
 By the canonical trivialization of $T\R^{2}$, we may consider that
 $TC=C\times\R^{2}$.  By the assumption, we have a trivialization of
 $TX=TM|X$ which restricts to the trivialization of $TW=TN(L)|W$
 determined by the framing $\nu$.  Thus, $d(\pi |W)$ is represented as 
\[
 TW\cong W\times\R^{3}\rightarrow C\times\R^{2}\cong TC; (x, (v_{1},
 v_{2}, v_{3}))\mapsto (\pi (x), (v_{2}, v_{3}))
\]
 Therefore, in
 order to obtain an extension of $d(\pi |W)$, we only have to show that
 the map $\pi |W:W\rightarrow C$ extends to $X$.  For this purpose,
 we may consider the problem up to homotopy.  Since $W$ (resp. $C$) is homotopy
 equivalent to $\partial N_{1/2}(L)=\partial X$ (resp. $S^{1}$), the
 projection $\pi |W$ determines a homotopy class $[\pi |W]\in [\partial
 X, S^{1}]$.  We will give an extension of $[\pi |W]$ in $[X, S^{1}]$.
 Since $S^{1}$ is the Eilenberg-MacLane space $K(\Z, 1)$, there are
 natural bijections $[X, S^{1}]\rightarrow H^{1}(X;\Z )$ and $[\partial
 X, S^{1}]\rightarrow H^{1}(\partial X;\Z )$ which commute the
 restriction maps (cf. Spanier \cite{Spanier}).  Combining these maps
 with Poincar\'{e}-Lefschetz duality (see Massey \cite{Massey} for the
 locally finite homology version), we have the following
 sign-commutative diagram.  
\begin{equation}
 \begin{array}{ccccc}
  [X, S^{1}] & \rightarrow & [\partial X, S^{1}] & & \\
  \parallel & & \parallel & & \\
  H^{1}(X;\Z ) & \rightarrow & H^{1}(\partial X;\Z ) &
   \stackrel{\delta}{\rightarrow} & H^{2}(X, \partial X;\Z
   ) \\
  \downarrow & & \downarrow & & \downarrow \\
  H_{2}^{\infty}(X, \partial X;\Z ) &
   \stackrel{\partial}{\rightarrow} & H_{1}(\partial X;\Z ) &
   \rightarrow & H_{1}^{\infty}(X;\Z ) 
 \end{array}
\end{equation}
 Here, the horizontal rows are cohomology and homology exact sequences and
 the vertical arrows are Poincar\'{e}-Lefschetz duality isomorphisms.
 As noted above, we may consider that $[\pi |W]$ belongs to $[\partial X,
 S^{1}]=H^{1}(\partial X;\Z )$.  We claim that $\delta [\pi |W]=0$.
 Through the Poincar\'{e}-Lefschetz duality $[\pi |W]$ corresponds to
 the homology class represented by the fiber of $\pi $ in
 $H_{1}(\partial X;\Z )$.  Consequently the class $[\pi |W]$
 corresponds to the class represented by the union of longitudes of
 $N(L)$.  Since the longitudes are preferred, the class vanishes in
 $H_{1}^{\infty}(X;\Z )$ which implies that $\delta [\pi |W]=0$.
 Hence, by the exactness of the sequence, we have a class in
 $H^{1}(X;\Z )=[X, S^{1}]$ which restricts to $[\pi |W]$.  
\end{proof} 
%%%%%%%% Remark of correction %%%%%%%%%%%%%%%%%%%%%%%%%%%%%%%%%%%%%%%%%
\begin{AppendixRem}
 This extension lemma does not hold under the condition of the original
 incorrect theorem in \cite{Top_paper}.  In fact, the consideration of
 framings of the tangent bundles was necessary.  
\end{AppendixRem}
%%%%%%%%%%%%%%%%%%%%%%%%%%%%%%%%%%%%%%%%%%%%% Remark of correction %%%%
 In the diagram (\ref{h-principle_diagram}) the differential map
 $d:\mathrm{Sbm}(X, C)\rightarrow\mathrm{Max}(TX, TC)$ is a
 weak homotopy equivalence.  Therefore, there exists
 $\psi\in\mathrm{Sbm}(X, C)$ such that $d\psi$ is homotopic to
 $\Phi$ in $\mathrm{Max}(TX, TC)$.  Thus $d\rho (\psi)=\rho
 (d\psi)$ is homotopic to $d(\pi |W)$ in $\mathrm{Max}(TW, TC)$.
 Since the differential map is a weak homotopy equivalence, 
 $\rho (\psi)$ and $\pi |W$ are regularly homotopic.  Moreover,
 the restriction map $\rho :\mathrm{Sbm}(X,
 C)\rightarrow\mathrm{Sbm}(W, C)$ is a fibration, thus the regular
 homotopy from $\rho (\psi)$ to $\pi |W$ is covered by a regular homotopy
 of $\psi$.  Hence we conclude that there exists
 $\varphi\in\mathrm{Sbm}(X, C)$ whose restriction to $W$ is
 $\pi |W$.  This completes the proof of Theorem \ref{correct_thm}.  
\end{proof}
%%%%%%%%%%%%%%%%%%%%%%%%%%%%%%%%%%%%%%%%%%%%%%%%%%%%%%%%%%%%%%%%%%%%%%%%%
%%%%%%%%%%%%%%%%%%%%%%%%%%%%%%%%%%%%%%%%%%% Proof of Theorem A %%%%%%%%%%
%%%%%%%%%%%%%%%%%%%%%%%%%%%%%%%%%%%%%%%%%%%%%%%%%%%%%%%%%%%%%%%%%%%%%%%%%
Finally, we give the proof of  Theorem \ref{thm_B}.  In the proof below, the
description of cutting open the residual components is improved in comparison
with the proof in \cite{Top_paper}.  
%%%%%%%%%%%%%%%%%%%%%%%%%%%%%%%%%%%%%%%%%%%%%%%%%%%%%%%%%%%%%%%%%%%%%%%%%
%%%%%%% Proof of Theorem B %%%%%%%%%%%%%%%%%%%%%%%%%%%%%%%%%%%%%%%%%%%%%%
%%%%%%%%%%%%%%%%%%%%%%%%%%%%%%%%%%%%%%%%%%%%%%%%%%%%%%%%%%%%%%%%%%%%%%%%%
\begin{proof}[Proof of Theorem \ref{thm_B}]
 Let $L$ be any $n$-component link in $M$.  Choose a framing $\nu
 :\bigsqcup_{j=1}^{n}(S^{1}\times D^{2})_{j}\rightarrow N(L)$.  By
 twisting the framing of a component once in the meridional direction if necessary, we may
 assume that there exists a trivialization of $TM$ whose restriction to
 $N(L)$ is equal to the trivialization induced by the chosen framing of
 $L$.  Let $\pi :N(L)\rightarrow D^{2}$ be the projection defined as in the
 proof of Theorem \ref{correct_thm}.  Now we consider the following
 commutative diagram.  
\begin{equation}\label{thm_B_diagram}
 \begin{array}{ccc}
  \mathrm{Sbm}(M, \R^{2}) & \stackrel{d}{\rightarrow} & \mathrm{Max}(TM,
   T\R^{2}) \\
  \rho\downarrow & & \rho\downarrow \\
  \mathrm{Sbm}(N(L), \R^{2}) & \stackrel{d}{\rightarrow} &
   \mathrm{Max}(TN(L), T\R^{2})
 \end{array}
\end{equation}
Here, the restriction maps $\rho$ are fibrations by Lemma
 \ref{Phillips_byproduct} and the differential maps $d$ are weak
 homotopy equivalences by Theorem \ref{Phillips_OM_thm} and
 \ref{Phillips_HD_thm}.  We claim the existence of an extension of $d\pi
 $.  
%%%%%%% Extension Claim for Theorem B %%%%%%%%%%%%%%%%%%%%%%%%%%%%%
\begin{AppendixClaim}
 There exists $\Phi\in\mathrm{Max}(TM, T\R^{2})$ such that $\rho
 (\Phi)=d\pi $.  
\end{AppendixClaim}
%%%%%%%%%%%%%%%%%%%%%%%%% Ext. Claim for Thm B %%%%%%%%%%%%%%%%%%%%
\begin{proof}[Proof of Claim]
 Since the trivialization of $TM|N(L)$ induced by the chosen framing of
 $L$ is the restriction of a trivialization of $TM$, as in the proof of
 the claim in the proof of Theorem \ref{correct_thm}, we only have to
 show that the map $\pi $ extends to $M$ up to homotopy.  However, since
 $\R^{2}$ is contractible this is clear.  
\end{proof}
\noindent
 Now, chasing the diagram (\ref{thm_B_diagram}) in the same way as in
 the proof of Theorem~\ref{correct_thm} shows that there exists an
 extension $\hat{\varphi}\in\mathrm{Sbm}(M, \R^{2})$ of $\pi $.  If the
 union of compact components of $\hat{\varphi}^{-1}(0)$ is equal to $L$,
 then set $\varphi:=\hat{\varphi}$ and we are done.  Otherwise, let $R$
 denote the union of compact components of $\hat{\varphi}^{-1}(0)$ which
 are not contained in $L$.  Note that $R$ has at most countably
 infinitely many components.  We will cut open these residual circles $R$
 by curves ending to ends of $M$.  It suffices to consider the case
 that there are infinitely many components of $R$.  The proof in the
 case of only finitely many components is similar and simpler.  Note that the
 components of $R$ cannot accumulate.   Now we fix an increasing
 filtration by codimension 0 compact connected submanifolds $N_{k}$ of
 $M\ (k\in\Z_{\geq 0})$ such that $\cup_{k=0}^{\infty}N_{k}=M$.  (We may
 assume that $M$ is connected.)  Suppose
 that we choose a decreasing filtration by open subsets $U^{e}_{k}$,
 each of which is a component of $M\setminus N_{k}$  for $k\in\Z_{\geq
 0}$.  Then it defines an end $e$ of $M$.  Here, we assume that
 $\partial N_{k}\cap R=\emptyset$ for any $k\in\Z_{\geq 0}$.  Let
 $\mathcal{E}$ denote the subset of the end set of $M$ consisting of all
 ends $e=\{ U^{e}_{k}\}_{k\in\Z_{\geq 0}}$ such that $U^{e}_{k}\cap
 R\neq\emptyset$ for any $k\in\Z_{\geq 0}$.  Since $\mathcal{E}$ is at
 most a countable set, we index it by natural numbers: $\mathcal{E}=\{
 e_{m}\}_{m\in\N}$.  Also, since there are at most countably many
 components of $R$, we number them as follows.  First, number the
 components of $R\cap N_{0}$ as $R_{1}\sqcup R_{2}\sqcup\cdots\sqcup
 R_{\ell_{0}}$, next $R\cap (N_{1}\setminus\Int
 N_{0})=R_{\ell_{0}+1}\sqcup\cdots\sqcup R_{\ell_{1}}$, and inductively
 $R\cap (N_{k}\setminus\Int N_{k-1})=R_{\ell_{k-1}+1}\sqcup\cdots\sqcup
 R_{\ell_{k}}$ for $k\in\N$.  

 We then define, inductively, simple curves $\alpha_{m}:[0,
 \infty)\rightarrow M\ (m\in\Z_{\geq 0})$ which cut $R$ open.  
 First, for the end $e_{1}=\{
 U^{1}_{k}\}_{k\in\Z_{\geq 0}}\in\mathcal{E}$, the sequence of the
 components of $R\cap(\cup_{k=0}^{\infty}(U^{1}_{k}\cap N_{k+1}))$ is an infinite
 subsequence of the components of $R$, which tends to the end $e_{1}$.
 Then we choose a simple curve $\alpha_{1}$ in
 $\cup_{k=0}^{\infty}(U^{1}_{k}\cap N_{k+1})$ which passes through one
 point in each circle of $R\cap(\cup_{k=0}^{\infty}(U^{1}_{k}\cap
 N_{k+1}))$ and tends to $e_{1}$.  Here, we choose $\alpha_{1}$ so that
 it passes through $R_{\ell}$ in order with respect to the indices
 $\ell$ of the circle $R_{\ell}$.  Set
 $R^{(1)}:=R\cap(\cup_{k=0}^{\infty}(U^{1}_{k}\cap N_{k+1}))$. 
 Inductively, for the end $e_{m}=\{ U^{m}_{k}\}_{k\in\Z_{\geq
 0}}\in\mathcal{E}$, the sequence of the components of
 $R^{(m)}:=(R\setminus\cup_{i=1}^{m-1}R^{(i)})\cap(\cup_{k=0}^{\infty}(U^{m}_{k}\cap
 N_{k+1}))$ is an infinite subsequence of circles of $R$ and we choose a
 simple curve $\alpha_{m}$ in $\cup_{k=0}^{\infty}(U^{m}_{k}\cap
 N_{k+1})$ which passes through $R^{(m)}$ in order and tends to $e_{m}$.
 Moreover, we choose all the curves $\alpha_{m}$ so that they do not
 intersect $L$ and are mutually disjoint.  Note that
 $R^{(0)}:=R\setminus\cup_{m=1}^{\infty}R^{(m)}$ is compact.  Thus, the
 components of $R^{(0)}$ are finitely many circles and we can easily
 choose a simple curve $\alpha_{0}$ which passes through those circles
 and tends to an end of $M$.  As is similar to the case of $\alpha_{m}$
 above, we take $\alpha_{0}$ so that it does not intersect $L$ nor
 $\alpha_{m}\ (m\in\N)$.   Now we claim the following.  
%%%%%% Second Claim for Theorem B %%%%%%%%%%%%%%%%%%%%%%%%%%%%%%%%%%%%%
\begin{AppendixClaim}
 $(M\setminus\cup_{m=0}^{\infty}\mathrm{Im}(\alpha_{m}), L)$ is
 diffeomorphic to $(M, L)$.  
\end{AppendixClaim}
%%%%%%%%%%%%%%%%%%%%%%%%%%%%%%%%%%%%%%%% Second Claim for Thm B %%%%%%%
\begin{proof}
 Set $P:=D^{2}\times [0, \infty)$ and let $\alpha :[0,
 \infty)\rightarrow P$ be the curve defined by $\alpha (t):=(0, t+1)$.
 Then we can easily construct a diffeomorphism between $P$ and
 $P\setminus\mathrm{Im}(\alpha )$ which is the identity near the
 boundary $(D^{2}\times\{ 0\})\cup (\partial D^{2}\times [0, \infty))$.
 By the construction of $\alpha_{m}$ the set $\{\alpha_{m}(0)\}$ is
 discrete in $M$ and the curves $\{\alpha_{m}\}$ do not accumulate.
 Hence, the claim follows.  
\end{proof}
\noindent
 Setting $\varphi :=\hat\varphi 
 |(M\setminus\cup_{m=0}^{\infty}\mathrm{Im}(\alpha_{m}))$, we have the desired
 submersion.  This completes the proof of Theorem \ref{thm_B}
\end{proof}

%%%%% Wearehere %%%%%%%%

\begin{acknowledgements}\label{ackref}
The author would like to express his hearty gratitude to Gilbert Hector
 who informed him that his earlier work contained an error.  He also thanks
 Daniel Peralta-Salas for studying, with G. Hector,  the problem which
 the author considered before and was caused by the author's
 misunderstanding.  Without their notice the author would never have found 
 the correction and new results in the point of view of the author.  
\end{acknowledgements}


\begin{thebibliography}{}
\baselineskip=14pt
\bibitem[HP]{Hector-Peralta}
{\bf G. Hector,  D. Peralta-Salas}, Integrable embeddings and
	foliations, {\em Amer. J. Math.} {134} (2012), 773--825.
%
\bibitem[Ms]{Massey}
{\bf W. S. Massey}, {\em Homology and cohomology theory}, Marcel Dekker
	    Inc., 1978.
%
\bibitem[My]{Top_paper}
{\bf S. Miyoshi}, Links and globally completely integrable vector
	fields on an open 3-manifold, {\em Topology} {34} (1995) 383--387.
%
\bibitem[P]{Phillips}
{\bf A. Phillips}, Submersions of open manifolds, {\em Topology}
	{6} (1966) 171--295.
%
\bibitem[R]{Rolfsen}
{\bf D. Rolfsen}, {\em Knots and Links}, AMS Chelsea publ., 1976. 
%
\bibitem[S]{Spanier}
{\bf E. H. Spanier}, {\em Algebraic Topology}, Springer, 1995.
%
\bibitem[W]{Whitehead}
{\bf J. H. C. Whitehead}, The immersion of an open 3-manifold in
	Euclidean 3-space, {\em Proc. London Math. Soc.} {11} (1961) 81--90.
%
\end{thebibliography}
\end{document}